\documentclass[a4paper,10pt]{scrartcl}

\usepackage[fleqn]{amsmath}
\usepackage{amssymb}
\usepackage{amsthm}
\usepackage{stmaryrd}
\usepackage{bbm}
\usepackage{hyperref}
\usepackage{mathtools}
\usepackage{color}

\definecolor{orange}{rgb}{1,0.5,0}

\mathtoolsset{showonlyrefs}

\newcommand{\bN}{\ensuremath{\mathbb{N}}}

\newcommand{\bR}{\ensuremath{\mathbb{R}}}

\newcommand{\ind}{\ensuremath{\mathbbm{1}}}

\newcommand{\cE}{\ensuremath{\mathcal{E}}}
\newcommand{\cF}{\ensuremath{\mathcal{F}}}
\newcommand{\cG}{\ensuremath{\mathcal{G}}}

\newcommand{\abs}[1]{\left\vert \, #1 \, \right\vert}
\newcommand{\norm}[1]{\left\Vert \, #1 \, \right\Vert}

\newcommand{\ddx}[1][1]{\ifnum#1=1 \frac{d}{dx} \else \frac{d^{#1}}{dx^{#1}} \fi}
\newcommand{\ddy}[1][1]{\ifnum#1=1 \frac{d}{dy} \else \frac{d^{#1}}{dy^{#1}} \fi}
\newcommand{\ddt}[1][1]{\ifnum#1=1 \frac{d}{dt} \else \frac{d^{#1}}{dt^{#1}} \fi}

\newtheorem{theorem}{Theorem}[section]
\newtheorem{lemma}[theorem]{Lemma}

\newcommand{\hmu}{\widehat{\mu}}
\newcommand{\omu}{\overline{\mu}}
\newcommand{\hnu}{\widehat{\nu}}

\DeclareMathOperator*{\essinf}{ess\,inf}
\DeclareMathOperator*{\esssup}{ess\,sup}

\newcommand{\TV}{\mathrm{TV}}

\begin{document}

\title{On maximal agreement couplings}

\author{Florian V\"ollering\footnote{Institut f\"ur Mathematische Statistik,
Westf\"alische Wilhelms-Universit\"at M\"unster,
Einsteinstra\ss{}e 62}}

\maketitle

\begin{abstract}
We call a coupling of two stochastic processes which maximizes the time until the first disagreement a maximal agreement coupling. We show that such a coupling always exists. Furthermore, it is possible to construct a lower bound on the disagreement time which is independent of one of the two processes.
\end{abstract}

\section{Introduction and Results}
Let $(E,\cE)$ be a Polish space equipped with the Borel $\sigma$-algebra.
Let $(Z^1_t)_{t\in\bN}, (Z^2_t)_{t\in\bN}$ be two $E$-valued stochastic processes on the canonical path space $(E^\bN,\cE^\bN)$ with laws $\mu^1$,$\mu^2$. We simply write $Z=(Z_t)_{t\in\bN}$ for a generic element of $E^\bN$.

A coupling of the measures $\mu^1$ and $\mu^2$ is a measure $\hmu$ on the product space $E^\bN\times E^\bN$ where the marginals are given by $\mu^1$ and $\mu^2$.

For a sub-$\sigma$-algebra $\cF\subset \cE^\bN$, denote the total variation distance with respect to $\cF$ by
\begin{align}
\norm{\mu^1-\mu^2}_{\cF-\TV} := \sup_{A\in \cF}(\mu^1(A)-\mu^2(A)).
\end{align}

A classical question is how quickly $Z^1$ and $Z^2$ can be coupled, that is finding a coupling under which the last time $Z^1$ and $Z^2$ disagree  is as small as possible.
More formally, let
\[ \sigma_0:=\inf\{t\geq 0: Z^1_s=Z^2_s \;\forall s\geq t\} \]
and $\cG_t:=\sigma(Z_s : s\geq t)$.
For any possible coupling $\hmu$ the coupling inequality
\begin{align}\label{eq:coupling-inequality}
\hmu(\sigma_0 \geq t) \geq \norm{\mu^1-\mu^2}_{\cG_t-\TV},
\end{align}
provides a universal lower bound. A \emph{maximal} coupling is a coupling for which \eqref{eq:coupling-inequality} is an equality for all $t\in\bN$,
and it is well-known that such a coupling always exists \cite{GRIFFEATH:75,GOLDSTEIN:79, THORISSON:86}.

We are interested in the opposite question, namely we want to find a coupling so that the \emph{first disagreement time} or \emph{decoupling time}
\begin{align}
\sigma:=\inf\{t\geq 0: Z^1_t\neq Z^2_t\}
\end{align}
is as big as possible. There is a corresponding coupling inequality for this question as well. Let $\cF_t:=\sigma(Z_s: 0\leq s\leq t)$.
\begin{lemma}\label{lemma:coupling-inequality}
For any coupling $\hmu$ of $\mu^1$ and $\mu^2$,
\begin{align}\label{eq:coupling-inequality2}
\hmu(\sigma > t) \leq 1 - \norm{\mu^1-\mu^2}_{\cF_t-\TV}\qquad\forall\;t\in\bN.
\end{align}
\end{lemma}
We call a coupling for which \eqref{eq:coupling-inequality2} is sharp for all $t\in\bN$ a \emph{maximal agreement coupling}. 
\begin{theorem}\label{thm:maximal-agreement-coupling}
There exists a maximal agreement coupling of $\mu^1$ and $\mu^2$. 
\end{theorem}
In general such a coupling is not unique, since there are no conditions on the joint distribution of $Z^1$ and $Z^2$ after the decoupling time $\sigma$. In fact, any coupling of the marginals after the decoupling time can be used to construct a maximal agreement coupling. Of course for this to be of use we need to describe the marginals first.

To this end we use the language of regular conditional probabilities. Fix $t\in\bN$, $i\in\{1,2\}$. Since $E$ is a Polish space regular conditional probabilities of $Z^i$ given the first $t+1$ steps exist. For $z\in E^{t+1}$ we we write $\mu^i(\cdot | Z=z)$ or $\mu^i(\cdot | z)$ for the regular conditional law of $Z^i$ given $Z^i_{0,...,t}=z$. We adopt similar notation for the regular conditional probabilities of other probability measures, in particular for couplings.
\begin{theorem}\label{thm:maximal-agreement-marginals}
Let $\hmu$ be a maximal agreement coupling.
\begin{enumerate}
\item
For $t\in\bN$, $s\geq t$, $z\in E^{s}$, $i=1,2$, the marginals after the decoupling time are given by
\[  \hmu(Z^i\in\cdot\;|\; Z^i=z, \sigma\geq t) = \mu^i(\;\cdot\;|\;z) \qquad \hmu-\text{a.s.} \]
\item 
For $t\in\bN$ and $z,z'\in E^{t+1}$ with $z_{0,...,t-1}=z'_{0,...,t-1}$ let $\hmu^{z,z'}_t$ be a coupling of $\mu^1(\cdot|z)$ and $\mu^2(\cdot|z')$. Assume that the map $(Z^1,Z^2,\sigma)\mapsto \hmu^{Z^1_{0,...,\sigma}, Z^2_{0,...,\sigma}}_\sigma$ is measurable. Then
\[ \hmu':= \int \hmu^{Z^1_{0,...,\sigma}, Z^2_{0,...,\sigma}}_\sigma\;d\hmu \]
is a maximal agreement coupling.
\end{enumerate}
\end{theorem}

It is clear that the event $\{\sigma=t\}$ contains information about $Z^1$ and $Z^2$. This is unavoidable, but also undesirable. In particular properties of the first disagreement time $\sigma$ cannot assumed to be stable under conditioning: $\hmu(\sigma=\infty)=\inf_{t\in\bN}\hmu(\sigma\geq t)$ might be positive, but $\hmu(\sigma\geq t | Z^1\in A_t)\to0$ for a decreasing sequence of events $A_t$, $A_t\in \cF_t$.

The second main result of this article is a remedy to this problem. There exists a lower bound $\tau$ on $\sigma$ which is \emph{independent} of $Z^1$. With this independence there is no problem in the above example when using $\tau$ instead of $\sigma$.
\begin{theorem}\label{thm:strong-agreement-coupling}
For any maximal agreement coupling $\hmu$ of $\mu^1$ and $\mu^2$ there exists an extension $\hnu$ to $E^\bN\times E^\bN\times(\bN\cup\{\infty\})$ by an additional random variable $\tau\in\bN\cup\{\infty\}$  with the following properties:
\begin{enumerate}
\item $\tau$ is independent of $Z^1$;
\item $\sigma \geq \tau \quad \hnu$-a.s.;
\item $\kappa_t:=\hmu(\tau=t|\tau\geq t)=1-\essinf_{B,z}\left\{\frac{\mu^2(Z_t\in B | Z=z)}{\mu^1(Z_t\in B | Z=z)} \right\},$
 where the infimum is taken over $\mu^1$-a.e. $z\in E^t$ and $B\subset E$ with $\mu^1(Z_t\in B | Z=z)>0$.
\end{enumerate}
In particular, if $\kappa_t<1$ for all $t\in\bN$ and $\sum_{t=0}^\infty \kappa_t<\infty$, then $\hnu(\tau=\infty)>0$.
\end{theorem}
In the case that $E$ is countable the following lemma provides convenient bounds on $\kappa_t$:
\begin{lemma}
Assume $E$ is countable. Define for $i=1,2$
\[ \delta^{(i)}_t:=\inf_{t\in\bN, z\in E^{t-1}, e\in E}\left\{\mu^i(Z_t=e|Z=z) : \mu^1(Z_t=e|Z=z)>0 \right\}. \] 
Then $\kappa_t \leq 1-\delta^{(2)}_t$ and
\[ \kappa_t\leq (\delta^{(1)}_t)^{-1}\sup\{\mu^2(Z_t=e|Z=z)-\mu^1(Z_t=e|Z=z) : z\in E^{t-1}, e\in E\}. \]
\end{lemma}
We finish with two remarks. First, we address an (impossible) generalization of Theorem \ref{thm:strong-agreement-coupling}. Clearly, in the theorem the roles of $Z^1$ and $Z^2$ can be reversed, so that there is also a r.v. $\tau'$ with $\tau'\leq \sigma$ and $\tau'$ independent of $Z^2$. One might wonder if it is possible to construct a (non-degenerate) time $\tilde{\tau}$ which satisfies $\tilde{\tau}\leq \sigma$ and which is independent of $Z^1$ and independent of $Z^2$ (clearly it cannot be independent of both simultaneously). However, this is not possible, as the following argument shows: Let $f:E\to\bR$, $t\geq 0$. Then
\begin{align}
 \mu^1\left(f(Z_t)\right)-\mu^2\left(f(Z_t)\right) 
&= \hnu\left( f(Z^1_t)-f(Z^2_t) \right)\\
&= \hnu\left( f(Z^1_t)\ind_{\tilde\tau\leq t}-f(Z^2_t)\ind_{\tilde\tau\leq t} \right)	\\
&= \hnu\left( f(Z^1_t)\ind_{\tilde\tau\leq t}\right)-\hnu\left(f(Z^2_t)\ind_{\tilde\tau\leq t} \right)
\end{align}
By the assumed individual independence, this equals
\[ \hnu(\tilde\tau\leq t)\left(\mu^1\left(f(Z_t)\right)-\mu^2\left(f(Z_t)\right)\right), \]
which implies $\tilde\tau=0$ a.s.

For the final remark we consider applying the results to Markov chains. Let $X^1$ and $X^2$ be two Markov chains with the same transition kernel but possibly different starting points $x_1$ and $x_2$ on a Polish space $F$. Clearly a maximal coupling of $X^1$ and $X^2$ is trivial, $\sigma=\infty$ if $x_1=x_2$ and $\sigma=0$ otherwise. However, let $\phi:F\to E$ and consider $Z^i_t=\phi(X^i_t)$, $t\in\bN$, $i=1,2$. For these induced processes a maximal agreement coupling is both meaningful and interesting. For example $\phi$ could be a coarse-graining map or a projection on a lower-dimensional state space.

\section{Preliminaries and the proof of Lemma \ref{lemma:coupling-inequality}}
Before going into the proofs we need some more notation and concepts. We say $\nu$ is a (sub-)probability measures when the total mass $\abs{\nu}$ is less or equal to 1. For two sub-probability measures $\nu^1$ and $\nu^2$, we say $\nu^1\leq \nu^2$ if $\nu^1(A)\leq \nu^2(A)$ for any event $A$, or equivalently $\nu^1\ll \nu^2$ and $\frac{d\nu^1}{d\nu^2}\leq 1$. The minimum $\nu^1\wedge \nu^2$ is the largest sub-probability measure $\nu$ which satisfies $\nu\leq \nu^1$ and $\nu\leq \nu^2$. With $\nu|_{\cF_t}$ we denote the restriction of the measure $\nu$ to the $\sigma$-algebra $\cF_t$. 
For $t\in\bN$, $z\in E^{t+1}$, the regular conditional probability $\nu(\cdot|Z=z)$ of a sub-probability measure $\nu$ is the regular conditional probability of the probability measure $\nu/\abs{\nu}$. A consequence of this convention is that for an event $A$ with $\nu(A)>0$ we have $\nu(\cdot|Z=z,A)=\nu(\cdot, A|Z=z)$.

The proof of the coupling inequality in Lemma \ref{lemma:coupling-inequality} is a simple computation using the minimum of two measures.
\begin{proof}[Proof of Lemma \ref{lemma:coupling-inequality}]
By the maximality of $\nu_t:=\mu^1|_{\cF_t}\wedge \mu^2|_{\cF_t}$ we have that the measures $\mu^1|_{\cF_t}-\nu_t$ and $\mu^2|_{\cF_t}-\nu_t$ are mutually singular, and hence
\[ \norm{\mu^1-\mu^2}_{\cF_t-\TV}=1-\abs{\nu_t}. \]
Furthermore, for $i=1,2$ and any coupling $\hmu$ and $A\in\cF_t$,
\[ \hmu(Z^1_{0,...,t}\in A,\sigma>t)=\hmu(Z^2_{0,...,t}\in A,\sigma>t)\leq \mu^i(A),\]
which by the maximality of $\nu_t$ implies $\hmu(Z^i_{0,...,t}\in\cdot,\sigma>t)\leq \nu_t$, $i=1,2$.
Therefore
\[ \hmu(\sigma>t)\leq \abs{\nu_t} = 1- \norm{\mu^1-\mu^2}_{\cF_t-\TV}. \qedhere \]
\end{proof}

\section{Proofs of Theorems \ref{thm:maximal-agreement-coupling} and \ref{thm:maximal-agreement-marginals}}
The proof of Theorem \ref{thm:maximal-agreement-coupling} is an explicity construction. It uses the same strategy as the proof for the existence of a maximal coupling found in \cite{THORISSON:00}(Theorem 4.6.1). The key difference is that we work with the \emph{increasing} sequence of $\sigma$-algebras $(\cF_t)$. In contrast the construction of the maximal coupling makes use of the decreasing sequence $(\cG_t)$. This difference means an inductive argument from the largest $\sigma$-algebra downwards is not possible.
\begin{proof}[Proof of Theorem \ref{thm:maximal-agreement-coupling}]
We will iteratively define a sequence of sub-probability measures which will allow us to construct the coupling. We start by setting $\overline{\mu}^i_0:= \mu^i$, $i=1,2$ and $\pi_0:=\overline{\mu}^1_0|_{\cF_0}\wedge\overline{\mu}^2_0|_{\cF_0}$, the largest common component of the two measures on the $\sigma$-algebra $\cF_0$. Note that we can interpret $\pi_0$ as a sub-probability measure on $E$. Next we set $\overline{\mu}^i_1(\cdot):= \int_E \overline{\mu}^i_0(\cdot| z)\pi_0(dz)$, which is the extension of $\pi_0$ to a sub-probability measure on $E^\bN$ which satisfies $\overline{\mu}^i_1\leq \overline{\mu}^i_0$.
finally we set $\mu^i_0:=\overline{\mu}^i_0-\overline{\mu}^i_1$, $i=1,2$. Iterating, we define
\begin{align}
&\pi_t:=\overline{\mu}^1_t|_{\cF_t}\wedge\overline{\mu}^2_t|_{\cF_t}, \label{eq:def-pi}\\
&\overline{\mu}^i_{t+1}(\cdot):= \int_{E^{t+1}} \overline{\mu}^i_t(\cdot| z)\pi_t(dz), \label{eq:def-omut}\\
&\mu^i_t:=\overline{\mu}^i_{t}-\overline{\mu}^i_{t+1}. \label{eq:def-mut}
\end{align}
From the construction we immediately obtain that 
\begin{align}
&\mu^i= \overline{\mu}^i_0 \geq \overline{\mu}^i_1\geq ..., &&
\sum_{s=0}^t \mu^i_s= \mu^i-\overline{\mu}^i_{t+1}\leq \mu^i,	\\
&\mu^i_t|_{\cF_s} = \pi_{t-1}|_{\cF_s}-\pi_{t}|_{\cF_s},&&
\norm{\mu^1_t-\mu^2_t}_{\cF_s-\TV}=0,\quad 0\leq s<t.
\end{align}
As a consequence, we can define $\mu^i_\infty := \mu^i-\sum_{s=0}^\infty \mu^i_s\geq 0$. Furthermore, 
\begin{align}
\mu^i_\infty|_{\cF_t}=\left[\mu^i-\sum_{s=0}^t\mu^i_s\right]|_{\cF_t}+\sum_{s=t+1}^\infty \pi_s|_{\cF_t} = \pi_t+\sum_{s=t+1}^\infty \pi_s|_{\cF_t},
\end{align}
which shows that $\mu^1_\infty=\mu^2_\infty$.

To obtain a coupling, let $\hmu_0:=\mu^1_0\otimes \mu^2_0$, and
\begin{align}\label{eq:def-hmut}
\hmu_t := \int_{E^t} \mu^1_t(\cdot|z)\otimes\mu^2_t(\cdot|z)\;\mu^1_t|_{\cF_{t-1}}(dz), \quad 1\leq t\leq \infty,
\end{align}
where for $t=\infty$ we have the degenerate case with $z\in E^\bN$ and $\mu^i_\infty(\cdot|z)=\delta_z$.

Define $\hmu=\hmu_0+\hmu_1+...+\hmu_\infty$, for which a direct computation shows that the marginals are $\mu^1$ and $\mu^2$, hence $\hmu$ is a coupling. What remains to show is that is indeed a maximal agreement coupling.

First we will show that for all $t\in\bN\cup\{\infty\}$,
\begin{align}\label{eq:hmut-sigma-t}
\hmu(\cdot,\sigma=t)=\hmu_t(\cdot),
\end{align}
which is equivalent to $\hmu_t(\sigma\neq t)=0$ for all $t\in\bN\cup\{\infty\}$.
By construction $\hmu_t(\sigma<t)=0$, and
\begin{align}
\mu^1_t|_{\cF_t}\wedge\mu^2_t|_{\cF_t} = \left(\overline{\mu}^1_t|_{\cF_t}-\overline{\mu}^1_{t+1}|_{\cF_t}\right)\wedge\left(\overline{\mu}^2_t|_{\cF_t}-\overline{\mu}^2_{t+1}|_{\cF_t}\right) = \overline{\mu}^1_t|_{\cF_t}\wedge \overline{\mu}^2_t|_{\cF_t} - \pi_t =0.
\end{align}
Therefore $\hmu_t(\sigma\leq t)=|\hmu_t|$, the total mass of $\hmu_t$, and hence $\hmu_t(\sigma>t)=0$. 

With \eqref{eq:hmut-sigma-t} we can now verify that $\hmu$ is indeed a maximal agreement coupling:
\begin{align}\label{eq:TVar-1}
\norm{\mu^1-\mu^2}_{\cF_t-\TV} = 1-\abs{\mu^1|_{\cF_t}\wedge\mu^2|_{\cF_t}}=1-\abs{\overline{\mu}^1_{t+1}}
\end{align}
and by \eqref{eq:coupling-inequality2} and \eqref{eq:TVar-1}
\begin{align}
\norm{\mu^1-\mu^2}_{\cF_t-\TV}\leq \hmu(\sigma\leq t) = \sum_{s=0}^t|\hmu_s| = 1-\abs{\overline{\mu}^1_{t+1}} = \norm{\mu^1-\mu^2}_{\cF_t-\TV},
\end{align}
which shows that \eqref{eq:coupling-inequality2} is an equality for all $t$ and hence $\hmu$ is indeed a maximal agreement coupling.
\end{proof}

The proof of Theorem \ref{thm:maximal-agreement-marginals} is mostly a refinement of the construction of the maximal agreement coupling above.
We first show that various regular conditional probabilities of the building blocks of $\hmu$ can be expressed via $\mu^1$ and $\mu^2$.
\begin{lemma}\label{lemma:omu-rewrite}
In the construction of the maximal agreement coupling of Theorem \ref{thm:maximal-agreement-coupling}, it holds that 
$\overline{\mu}^i_t(\cdot | z ) = \mu^i(\cdot | z )$ for all $s\geq t$, $\omu^i_t|_{\cF_{s-1}}$-a.e. $z\in E^{s}$, and $\mu^i_t(\cdot | z ) = \mu^i(\cdot | z )$ for all $s\geq t$, $\mu^i_t|_{\cF_s}$-a.e. $z\in E^{s+1}$.
\end{lemma}
\begin{proof}
First we show that $\overline{\mu}^i_t(\cdot | z ) = \mu^i(\cdot | z )$ for $\overline{\mu}^i_t|_{\cF_s}$-a.e. $z\in E^s$, and the proof is done by induction. The claim is clearly true for $t=0$, since $\omu^i_0=\mu^i$. Assume now the claim is true for $t\in\bN$. Let $s\geq t+1$, $z\in E^{t+1}$ and $z'\in E^{s+1}$ with $z'_{0,...,t}=z$. Since $\overline{\mu}^i_{t+1}(\cdot )=\int_{E^{t+1}}\overline{\mu}^i_t(\cdot | \gamma) \pi_t(d\gamma)$ and $z\in E^{t+1}$ we have $\overline{\mu}^i_{t+1}(\cdot | z)= \overline{\mu}^i_{t}(\cdot | z)$ for $\pi_t$-a.e. $z\in E^{t+1}$. Since $\omu^i_{t+1}|_{\cF_t}=\pi_t\leq \omu^i_t|_{\cF_t}$, the induction hypothesis implies 
$\overline{\mu}^i_{t+1}(\cdot | z)= \mu^i(\cdot | z)$ for $\omu^i_{t+1}$-a.e. $z\in E^{t+1}$. To obtain the statement for $z'$ we use the fact that $\mu^i(\cdot|z')$ is a version of the regular conditional probability $\nu_z(\cdot|z')$, where $\nu_z=\mu^i(\cdot|z)$. 

For the second claim, let $s\geq t$, $A\in\cG_{s+1}$ and $B\in\cF_s$. Then, using the definition of $\mu^i_t$ and the first claim,
\begin{align}
&\int_B \mu^i_t(A|z)\mu^i_t|_{\cF_s}(dz) = \mu^i_t(A\cap B) = \omu^i_t(A\cap B) - \omu^i_{t+1}(A\cap B) \\
&= \int_B \mu(A|z)(\omu^i_t-\omu^i_{t+1})|_{\cF_s}(dz) = \int_B \mu^i(A|z)\mu^i_t|_{\cF_s}(dz).	\qedhere
\end{align}
\end{proof}
\begin{proof}[Proof of Theorem \ref{thm:maximal-agreement-marginals}]
Part a): First assume that $\hmu$ is the maximal agreement coupling constructed in Theorem \ref{thm:maximal-agreement-coupling}. 
By \eqref{eq:hmut-sigma-t}, \eqref{eq:def-omut} and Lemma \ref{lemma:omu-rewrite},
\[ \hmu(\cdot | Z^i=z, \sigma\geq t )=\hmu(\cdot, \sigma\geq t | Z^i=z)=\omu_t(\cdot|Z^i=z) = \mu^i_t(\cdot | z) = \mu^i(\cdot|z), \]
which shows the claim for this maximal agreement coupling. Assume now that $\hmu'$ is some other maximal agreement coupling. Define the sub-probability measure $\pi_t'(A):=\hmu'(Z^i\in A, \sigma>t)$, $A\in\cF_t$, $i=1,2$. The definition of $\pi'_t$ does not depend on the choice of $i$ since $\sigma>t$ and $A\in\cF_t$. Therefore $\pi'_t\leq \mu^i$ for $i=1$ and $i=2$, which implies $\pi'_t\leq \pi_t$. But by the maximal agreement property of $\hmu'$, $|\pi'_t|=|\pi_t|$, which implies $\pi'_t=\pi_t$. Defining $\omu^{\prime,i}_t(\cdot):=\hmu'(Z^i\in \cdot, \sigma>t)$ and $\mu^{\prime,i}_t=\omu^{\prime,i}_t-\omu^{\prime,i}_{t+1}$, the proof of Lemma \ref{lemma:omu-rewrite} and the above argument for $\hmu$ are true for $\hmu'$ as well, using only $\pi'_t=\pi_t$.

For part b), in \eqref{eq:def-hmut} we replace $\mu^1_t(\cdot|z)\otimes\mu^2_t(\cdot|z)$ by
\[ \int_{E\times E}\hmu^{(z,\gamma_1),(z,\gamma_2)}_t(Z^1\in\cdot, Z^2\in\cdot) \left[\mu^1_t(Z_t\in\cdot|z)\otimes\mu^2_t(Z_t\in\cdot|z)\right](d(\gamma_1,\gamma_2)).
\]
By Lemma \ref{lemma:omu-rewrite} the marginals stay the same, so we obtain a valid coupling of $\mu^1$ and $\mu^2$. And since the change affects only the evolution after the decoupling time, the maximal agreement property remains unaffected.
\end{proof}

\section{Proof of Theorem \ref{thm:strong-agreement-coupling}}
This proof relies on a refinement of the construction of the maximal agreement coupling in the previous section. The next lemma is the key ingredient. Basically, it is the analogous statement of Theorem \ref{thm:strong-agreement-coupling} for a single time point $t$.
\begin{lemma}\label{lemma:Y_t}
Fix $t\in\bN$. A maximal agreement coupling $\hmu$ of $\mu^1$ and $\mu^2$ can be extended to a coupling $\hmu^{Y_t}$ on $E^\bN\times E^\bN \times \{0,1\}$ containing an additional random variable $Y_t\in\{0,1\}$ with the following properties:
\begin{enumerate}
\item $\hmu^{Y_t}(Y_t=1)=\kappa_t$, where $\kappa_t$ is as in Theorem \ref{thm:strong-agreement-coupling};
\item $Y_t$ is independent of $Z^1$ and $\{\sigma>t-1\}$;
\item $\{\sigma=t\}\subset \{\sigma>t-1,Y_t=1\}$.
\end{enumerate}
\end{lemma}
\begin{proof}
Assume that $\kappa_t\in(0,1)$, otherwise the statement is trivial. Furthermore assume for now that $\hmu$ is the maximal agreement coupling constructed in the proof of Theorem \ref{thm:maximal-agreement-coupling}.
For $A\subset E^t$ and $B\subset E$, we write
\[ \kappa_t(A,B) := \hmu(\sigma=t|Z^1_t\in B, Z^1_{0,...,t-1}\in A, \sigma\geq t). \]
Since $\hmu(Z^1\in\cdot, \sigma=s)=\mu^1_s(\cdot)$ and $\mu^1_t=\omu^1_t-\omu^1_{t+1}$, we have
\begin{align}
\kappa_t(A,B) &= \frac{\hmu(\sigma=t,Z^1_t\in B, Z^1_{0,...,t-1}\in A)}{\hmu(\sigma\geq t,Z^1_t\in B, Z^1_{0,...,t-1}\in A)} = \frac{\mu^1_t(Z_t\in B, Z_{0,...,t-1}\in A)}{\omu^1_t(Z_t\in B, Z_{0,...,t-1}\in A)}\\
&= 1-\frac{\omu^1_{t+1}(Z_t\in B, Z_{0,...,t-1}\in A)}{\omu^1_t(Z_t\in B, Z_{0,...,t-1}\in A)}. \label{eq:kappa-4}
\end{align}
We want to show that $\kappa_t(A,B) \leq \kappa_t$. To this end, by \eqref{eq:def-omut} and Lemma \ref{lemma:omu-rewrite},
\begin{align}
&\omu^1_t(Z_t\in B, Z_{0,...,t-1}\in A) 
= \int_A \int_B 1\ \mu^1(Z_t\in dy|z)\pi_{t-1}(dz)	\\
&= \int_A \int_B \frac{d\mu^1(Z_t\in \cdot|z)}{d\mu^2(Z_t\in \cdot|z)}(y) \mu^2(Z_t\in dy|z)\pi_{t-1}(dz),
\end{align}
where we used in the last line that $\mu^1(Z_t\in \cdot|z)\ll\mu^2(Z_t\in \cdot|z)$ (for a.e. $z$) since $\kappa_t<1$.
By using the fact that for any $a\in\bR$, $a=(a\wedge 1)(a\vee 1)$, we can upper bound the above by
\begin{align}
&\esssup_{z\in A, y\in B} \left(\frac{d\mu^1(Z_t\in \cdot|z)}{d\mu^2(Z_t\in \cdot|z)}(y)\vee 1\right) \int_A \int_B \frac{d\mu^1(Z_t\in \cdot|z)}{d\mu^2(Z_t\in \cdot|z)}(y)\wedge 1\; \mu^2(Z_t\in dy|z)\pi_{t-1}(dz)	\\
&= \esssup_{z\in A, y\in B} \left(\frac{d\mu^1(Z_t\in \cdot|z)}{d\mu^2(Z_t\in \cdot|z)}(y)\right) \int_A \int_B \left[\mu^1(Z_t\in \cdot|z)|_{\cF_t}\wedge \mu^2(Z_t\in\cdot|z)|_{\cF_t}\right](dy) \pi_{t-1}(dz)	\\
&\leq (1-\kappa_t)^{-1} \pi_t(Z_t\in B, Z_{0,...,t-1}\in A) = (1-\kappa_t)^{-1} \omu^1_{t+1}(Z_t\in B, Z_{0,...,t-1}\in A),
\end{align}
where in the last line we used \eqref{eq:def-pi} and \eqref{eq:def-omut}.
It follows that \eqref{eq:kappa-4} is indeed less or equal to $\kappa_t$.
Define now for $z\in E^{t+1}$ $\kappa_t(z):=\hmu(\sigma=t|Z^1=z, \sigma\geq t)$. Since $\kappa_t(A,B)\leq \kappa_t$ for all $A,B$ we have also that $\kappa_t(z)\leq \kappa_t$ for $\omu^1_{t}$-a.e. $z\in E^{t+1}$.

We can define the extended coupling $\hmu^{Y_t}$ on $E^\bN\times E^\bN\times\{0,1\}$ via $\hmu^{Y_t}_s=\hmu_s\otimes(\kappa_t\delta_1+(1-\kappa_t)\delta_0)$, $s<t$, $\hmu^{Y_t}_t= \hmu_t\otimes \delta_1$ and
\begin{align}
\hmu^{Y_t}_s=\int_{E^{s}}\hmu_s(\cdot|Z^1=Z^2=z)\otimes\left(\left(1-\frac{1-\kappa_t}{1-\kappa_t(z_{0,...,t})}\right)\delta_1+\frac{1-\kappa_t}{1-\kappa_t(z_{0,...,t})}\delta_0\right)\mu^1_s|_{\cF_{s-1}}(dz)
\end{align}
for $s>t$, and we set $\hmu^{Y_t}=\hmu^{Y_t}_0+...+\hmu^{Y_t}_\infty$.

What remains is to verify that properties a), b) and c) hold.
Property c) follows from $\hmu^{Y_t}(\cdot, \sigma=s)=\hmu^{Y_t}_s$ and the definition of $\hmu^{Y_t}_t$. For a) and b), let $A\in\cF_t$. By the construction of $\hmu^{Y_t}$,
\begin{align}
&\hmu^{Y_t}(Y_t=1,Z^1\in A,\sigma\geq t) = \hmu^{Y_t}_t(Y_t=1,Z^1\in A)+...+\hmu^{Y_t}_\infty(Y_t=1,Z^1\in A)\\
&= \mu^1_t(Z^1\in A) + \int_A\left(1-\frac{1-\kappa_t}{1-\kappa_t(z)}\right)\omu^1_{t+1}|_{\cF_t}(dz).	\label{eq:kappa-2}
\end{align}
By \eqref{eq:kappa-4} and $\kappa_t<1$, 
\[ (1-\kappa_t(A,B))^{-1}=\frac{\omu^1_t(Z_t \in B, Z_{0,...,t-1}\in A)}{\omu^1_{t+1}(Z_t \in B, Z_{0,...,t-1}\in A)}<\infty, \]
from which follows that $\omu^1_t|_{\cF_t} \ll \omu^1_{t+1}|_{\cF_t}$ and 
\begin{align}\label{eq:density}
\frac{d\omu^1_t|_{\cF_t}}{d\omu^1_{t+1}|_{\cF_t}}(z)=(1-\kappa_t(z))^{-1}.
\end{align}
Together with \eqref{eq:kappa-2} we obtain
\begin{align}
\hmu^{Y_t}(Y_t=1 ,Z^1\in A, \sigma\geq t ) 
&= \mu^1_t(A) + \omu^1_{t+1}(A)-(1-\kappa_t)\omu^1_t(A) \\
&= \kappa_t\omu^1_t(A)\\
&= \kappa_t\hmu^{Y_t}(Z^1\in A,\sigma\geq t). \label{eq:kappa-3}
\end{align}
This shows both that $\hmu^{Y_t}(Y_t=1)=\kappa_t$ and independence of $Z^1_{0,...,t}$ and $\{\sigma \geq t\}$. 
To obtain the full independence of $Z^1$, let $B\in\sigma(Z^1_{t+1},...,Z^1_s)$ for $s>t$ arbitrary. Then, by \eqref{eq:def-hmut} and Lemma \ref{lemma:omu-rewrite}, \eqref{eq:kappa-2} changes to
\begin{align}
&\hmu^{Y_t}(Y_t=1,Z^1\in A\cap B,\sigma\geq t)  \\
&= \mu^1_t(A\cap B) + \int_A\left(1-\frac{1-\kappa_t}{1-\kappa_t(z)}\right)\mu^1(B|z)\omu^1_{t+1}|_{\cF_t}(dz).	
\end{align}
With the same computation as in \eqref{eq:kappa-3} we get
\[ \hmu^{Y_t}(Y_t=1,Z^1\in A\cap B,\sigma\geq t) =\kappa_t\hmu^{Y_t}(Z^1\in A\cap B,\sigma\geq t), \]
which completes the proof for $\hmu$.

To show the statement for a general maximal agreement coupling $\hmu$ we use the same strategy as in the proof of Theorem \ref{thm:maximal-agreement-marginals}. We define $\pi_t$, $\omu^i_t$ and $\mu^i_t$ in terms of $\hmu$: 
\begin{align}\label{eq:hmu-general}
\begin{aligned}
\pi_t&:=\hmu(Z^1\in\cdot, \sigma>t)|_{\cF_t},	\\
\mu^i_t&:=\hmu_t(Z^i\in\cdot)	,
\end{aligned}
\qquad
\begin{aligned}
\hmu_t&:=\hmu(\cdot, \sigma=t)	,\\
\omu^i_t&:=\hmu_t(Z^i\in \cdot, \sigma\geq t).
\end{aligned}
\end{align}
We restate that $\pi_t$ is universal in maximal agreement couplings, as was shown in the proof of Theorem \ref{thm:maximal-agreement-marginals}. Using this the above construction of $\hmu^{Y_t}$ follows through the same.
\end{proof}
Theorem \ref{thm:strong-agreement-coupling} is a generalization of Lemma \ref{lemma:Y_t}, and the proof reflects this.
\begin{proof}[Proof of Theorem \ref{thm:strong-agreement-coupling}]
 We will introduce random variables $(Y_t)_{t\in\bN}$ in such a way that the law of $(Z^1,Z^2,Y_t)$ is given by the coupling $\hmu^{Y_t}$ constructed in Lemma \ref{lemma:Y_t}. We do this by using the way $\hmu_s$ is extended to $\hmu^{Y_t}_s$ simultaneously for all $Y_t$. For $s,t\in\bN$ and $z\in E^{s+1}$ let
\begin{align}
\nu_{s,t}(z):=
\begin{cases}
\kappa_t\delta_1+(1-\kappa_t)\delta_0, &s<t;\\
\delta_1, &s=t;\\
\left(1-\frac{1-\kappa_t}{1-\kappa_t(z_{0,...,t})}\right)\delta_1+\frac{1-\kappa_t}{1-\kappa_t(z_{0,...,t})}\delta_0, &s>t.
\end{cases}
\end{align} 
Note that $\nu_{s,t}(z)$ is the distribution of $Y_t$ given $\{\sigma=s\}$ and $Z^1=z$.
By simply taking the product measures we obtain a coupling $\hnu=\hnu_0+...+\hnu_\infty$,
\begin{align} 
\hnu_s = \int_{E^{s+1}}\hmu_s(\cdot|Z^1=z)\otimes\bigotimes_{t=0}^\infty \nu_{s,t}(z)\ \mu^1_s|_{\cF_s}(dz),
\end{align}
where $\mu^1_s$ and $\hmu_s$ are given by \eqref{eq:hmu-general}.
This construction indeed extends the maximal agreement coupling $\hmu$ by a sequence $(Y_t)_{t\in\bN}$ and the marginal of $(Z^1,Z^2,Y_t)$ is given by $\hmu^{Y_t}$.

Let $\tau:=\inf\{t\geq 0 : Y_t=1\}$. By construction, $\hnu_t(Y_t=1)=1$. This implies $Y_\sigma=1$ and hence $\tau\leq \sigma$ $\hnu$-a.s. Furthermore we get $\hnu_s(\tau=t)=0$ for all $t>s$.

Let $A\subset E^\bN$ be an arbitrary event. We have
\begin{align}
&\hnu(Z^1\in A, \tau>t)= \hnu(Z^1\in A, Y_{t}=0, \tau>t-1 ) \\
&= (\hnu_{t+1}+...+\hnu_\infty)(Z^1\in A, Y_{t}=...=Y_0=0 ).
\end{align}
For $r> t$, 
\begin{align}
&\hnu_r(Z^1\in A, Y_{t}=...=Y_0=0 ) \\
&= \int_{E^{r+1}}\hmu_r(Z^1\in A|Z^1=z)\otimes\bigotimes_{s=0}^{t} \left[\nu_{r,s}(z)\right](Y_s=0)\mu^1_r|_{\cF_r}(dz) \\
&= \int_{E^{r+1}}\mu^1_r(A|z)\prod_{s=0}^{t}\frac{1-\kappa_s}{1-\kappa_s(z_{0,...,s})}\mu^1_r|_{\cF_r}(dz).
\end{align}
By Lemma \ref{lemma:omu-rewrite}, $\mu^1_r(A|z)=\mu^1(A|z)$. 
Summing over $r> t$, we get
\begin{align}
&\hnu(Z^1\in A, \tau>t)=
\left(\prod_{s=0}^{t} (1-\kappa_s)\right) \int_{\mathrlap{E^{t+1}}}\ \mu^1(A|z)\prod_{s=0}^{t}\frac{1}{1-\kappa_s(z_{0,...,s})}\omu^1_{t+1}|_{\cF_{t}}(dz).
\end{align}
By \eqref{eq:density}, $(1-\kappa_{t}(z_{0,...,t}))^{-1}=\frac{d\omu^1_{t}|_{\cF_{t}}}{d\omu_{t+1}|_{\cF_{t}}}$. Together with Lemma \ref{lemma:omu-rewrite} this allows us to simplify the integral to
\[ \int_{\mathrlap{E^{t}}}\ \mu^1(A|z)\prod_{s=0}^{t-1}\frac{1}{1-\kappa_s(z_{0,...,s})}\omu^1_{t}|_{\cF_{t-1}}(dz).  \]
Repeating the argument shows that it in fact equals $\int_E\mu^1(A|z)\omu^1_0|_{\cF_0}(dz)=\mu^1(A)$, which shows that
\begin{align}
\hnu(Z^1\in A, \tau>t)
&= \left(\prod_{s=0}^{t} (1-\kappa_s)\right)\mu^1(A)
= \hnu(\tau > t)\hnu(Z^1\in A).
\end{align}
\end{proof}

\bibliography{BibCollection}{}
\bibliographystyle{plain}

\end{document}